\documentclass[12pt]{amsart}

\usepackage[margin=1.15in]{geometry}
\usepackage{amscd,amssymb, amsmath,amsfonts, wasysym, mathrsfs, mathtools,enumerate,hhline,color}
\usepackage{graphicx}
\usepackage[all, cmtip]{xy}

\usepackage{url}

\definecolor{hot}{RGB}{65,105,225}

\usepackage[pagebackref=true,colorlinks=true, linkcolor=hot ,  citecolor=hot, urlcolor=hot]{hyperref}

\stepcounter{section}

\theoremstyle{plain}
\newtheorem{theorem}{Theorem}[section]

\newtheorem{lemma}[theorem]{Lemma}
\newtheorem{thrm}[theorem]{Theorem}

\theoremstyle{definition}

\newtheorem{rmk}[theorem]{Remark}

\newtheorem*{ex*}{Example}

\def\bR{\mathbf{R}}

\newcommand\cc{{\mathbb{C}}}

\DeclareMathOperator{\codim}{codim}              
                  
\DeclareMathOperator{\id}{id}                    

\def\ra{\rightarrow}

\def\bC{\mathbb{C}}

\def\cH{\mathcal{H}}

\def\cO{\mathcal{O}}

\def\cL{\mathcal{L}}
\def\bH{\mathbb{H}}
\def\bZ{\mathbb{Z}}
\def\bW{\mathbf{W}}
\def\bV{\mathbf{V}}
\def\bM{\mathbf{M}}
\def\eps{\epsilon}

\newcommand{\ubul}{{\,\begin{picture}(-1,1)(-1,-3)\circle*{2}\end{picture}\ }}

\title[Rank one local systems and forms of degree one]{Rank one local systems and forms of degree one}
\begin{document}
\author{Nero Budur}
\email{Nero.Budur@wis.kuleuven.be}
\address{KU Leuven, Department of Mathematics,
Celestijnenlaan 200B, B-3001 Leuven, Belgium}

\author{Botong Wang}
\email{bwang3@nd.edu}
\address{University of Notre Dame, Department of Mathematics, 255 Hurley Hall, IN 46556, USA} 

\author{Youngho Yoon}
\email{mathyyoon@ibs.re.kr}
\address{Center for Geometry and Physics, Institute for Basic Science (IBS), Pohang 790-784, Republic of Korea}


\thanks{The first author was partly sponsored by NSA and FWO}
\thanks{The third author was supported by IBS-R003-G2.}

\begin{abstract}
Cohomology support loci of rank one local systems of a smooth quasi-projective complex algebraic variety are finite unions of torsion-translated complex subtori of the character variety of the fundamental group. Tangent spaces of the character variety are (partially) represented by logarithmic 1-forms. In this paper, we give a relation between cohomology support loci and the natural strata of 1-forms given by the dimension of the vanishing locus. This relation generalizes the one for the projective case due to Green and Lazarsfeld and also generalizes the partial relation due to Dimca in the quasi-projective case.
\end{abstract}

\maketitle

\section*{}

In this note we consider the relation between the cohomological properties of rank one local systems on a quasi-projective compex algebraic variety and the zero loci of logarithmic 1-forms. There are two distinct types of logarithmic 1-forms to consider, according to the divisor along which the poles are allowed: simple normal crossings divisor, or not. In this paper we deal with the first case.

For a smooth quasi-projective complex algebraic variety $U$, the space of rank one local systems is identified with the space of characters of the fundamental group:
$$
\bM_B(U)=Hom (\pi_1(U),\bC^*)=Hom (H_1(U,\bZ),\bC^*)=H^1(U,\bC^*).
$$
This is an algebraic group consisting of finitely many copies of $(\bC^*)^{b_1(U)}$. The cohomology support loci are defined as
$$
\bV^i=\{\rho \in \bM_B(U)\mid  H^i(U,L_\rho)\ne 0 \},
$$
where $L_\rho$ is the local system given by the rank one representation $\rho$.  $\bV^i$ are finite unions of torsion-translated subtori of $\bM_B(U)$, see \cite{BW1}.

Let $X$ be a smooth projective compactification of $U$ with complement $D$ a simple normal crossing divisor. Let
$$
\bW=H^0(X,\Omega_X^1(\log D)).
$$
Denote by $Z(w)$ the degenerating locus of a section $w\in \bW$. Stratify $\bW$ according to the dimension of $Z(w)$:
$$
\bW^i=\{ w\in \bW\mid \codim_X(Z(w))\le i\}.
$$
The sets $\bW^i$ are closed in $\bW$. Indeed, $Z(w)$ are fibers of the second projection from $X\times\bW$ restricted to the correspondence $Z=\{(x,w)\mid x\in Z(w)\}$, and hence Chevalley's upper-semicontinuity theorem for dimension of fibers of a proper morphism applies.

Our goal is to understand the relation between $\bV^i$ and $\bW^i$. Note that $\bW^i$ are algebraic invariants, whereas $\bV^i$ are homotopy invariants. Let us describe the main result.

The tangent space to $\bM_B(U)$ at a representation $\rho$ is 
$$
T_\rho(\bM_B(U))=H^1(U,End(L_\rho)).
$$
Every irreducible component of the tangent cone to $\bV^i$ at a torsion representation $\rho$,
$$
TC_\rho(\bV^i)\subset H^1(U,End(L_\rho)),
$$
is a vector subspace.  Consider now the map on $\bM_B(U)$ given by translation by $\rho$. It induces an isomorphism from the tangent space at the trivial representation to $T_\rho\bM_B(U)$. We will abuse the notation and call the inverse of this isomorphism $\rho^{-1}$. Thus,
$$
\rho^{-1}T_\rho(\bM_B(U))=T_1(\bM_B(U))=H^1(U,\bC).
$$
Note that 
$$
\bW=F^1H^1(U,\bC)=H^0(X,\Omega_X^1(\log D)),
$$
where $F$ is the Hodge filtration. Placing the collection of tangent cones at torsion representations in the same single vector space, we define
$$
\bR^i=\mathop{\bigcup_{k\le i}}_{\rho\in \bM_B(U)\text{ torsion }}\rho^{-1}TC_\rho(\bV^k) \subset H^1(U,\bC)
$$
and
$$
\bR_i=\mathop{\bigcup_{k\ge i}}_{\rho\in \bM_B(U)\text{ torsion }}\rho^{-1}TC_\rho(\bV^k) \subset H^1(U,\bC).$$

\begin{thrm}\label{thrmMain}
Let $U$ be a smooth quasi-projective complex algebraic variety of dimension $n$. Let $\bR^i$ be the collection of tangent cones at torsion points to the cohomology support loci as above. Let $\bW^i$ be the strata of 1-forms, logarithmic with respect to a fixed good compactification, given by the dimensions of the zero loci. Then
$$
(\bR^i\cup \bR_{2n-i})\cap \bW\subset \bW^i.
$$
\end{thrm}

We do not have an example where equality fails to hold in Theorem \ref{thrmMain}.  

If $U=X$ is projective, the result is due to Green-Lazarfeld \cite{GrL}.  In this case $\bR^i=\bR_{2n-i}$ and $\bW=H^0(X,\Omega^1_X)$. This has been recently used by Popa-Schnell \cite{PS} to prove that, if $X$ is of general type, then every holomorphic global 1-form must vanish at some point. That is, $\bW=\bW^n$ for $n=\dim X$. 

In the quasi-projective case, a result of Dimca \cite[Theorem 6.1]{D} implies the inclusion as in Theorem \ref{thrmMain}, but without the appearance of the tangent cones at torsion representations, and with $\bW$ replaced by $H^{1,0}\cup H^{1,1}$, where
$$
\bW=F^1H^1(U,\bC)=H^{1,0}\oplus H^{1,1}.
$$ 
In particular, we allow log forms of mixed type. It was noted in {\it loc. cit.} that there are positive-dimensional components of $\bW^i$ unaccounted for by the tangent cones at the trivial representation. Our result explains this phenomenon, for the existing examples, by allowing torsion representations as well. 

The case when the 1-forms are logarithmic with respect to a divisor with worse than simple normal crossing singularities is poorly understood. This case is highly interesting from an applied point of view too, such as for maximum likelihood estimation in statistics, see \cite{HS}. There are some partial results for complements of hyperplane arrangements, e.g. \cite{CDFV}. 

The main idea of the proof, due to \cite{GrL}, is to relate a complex computing local system cohomology with a Koszul type complex. In the projective case this is facilitated by the close relation between local systems and Higgs line bundles. In the quasi-projective case, \cite{D} used results of Arapura \cite{A} on Higgs fields. In contrast to the methods of \cite{GrL} and \cite{D}, we do not use any Hodge theory. Instead, we use properties of the cohomology jump ideals developed in \cite{BW2}.

\section{Proof of Theorem \ref{thrmMain}}

We keep the notation as in Theorem \ref{thrmMain}. We shall prove the following equivalent formulation.

\begin{thrm}\label{thrm2vs} Let $w\in H^0(X,\Omega^1_X(\log D))$. Suppose that $\codim_X(Z(w))>i$. Then for any torsion $\rho\in\bM_B(U)$ and $k\le i$ or $k\ge 2n-i$,  $\exp_\rho([w])$ is not contained in $TC_\rho(\bV^k)$.
\end{thrm}

Here $\exp_\rho$ is the isomorphism between $H^1(U,\bC)$ and $T_\rho(\bM_B(U))$, and $[w]$ denotes the class in $H^1(U,\bC)$. 

Let $$(\Omega_X^\ubul(\log D)\otimes_{\cO_X}\cL_\rho,d_\rho)$$ be the logarithmic de Rham complex of the local system $L_\rho$, where $\cL_\rho$ is the canonical Deligne extension. It suffices to prove that 
$$\bH^k(X,(\Omega_X^\ubul(\log D)\otimes_{\cO_X}\cL_\rho,d_\rho +tw))=0$$
for general $t\in\bC$, for $k\le i$ and for $k\ge 2n-i$, where from now $w$ is viewed as form with values in $End(\cL_\rho)\cong \cO_X$.

Let 
$$
N^\ubul=(\Omega_X^\ubul(\log D)\otimes_{\cO_X}\cL_\rho\otimes_\bC \bC[\eps,t],\eps d_\rho +tw).
$$
This is a complex of coherent $\cO_{X\times\bC^2}$-modules with differential only $\cO_{\bC^2}$-linear. Denote the projection $X\times\bC^2\ra \bC^2$ by $p$.

\begin{lemma}
$\bR^kp_*(N^\ubul)$ is a coherent sheaf on $\bC^2$ for every $k$.
\end{lemma}
\begin{proof}
There is a spectral sequence, essentially same as the first hypercohomology spectral sequence, with the maps in the spectral sequence $\cO_{\bC^2}$-linear,
$$
E^{pq}_1=\bR^qp_*(N^p) \Rightarrow \bR^{p+q}p_*(N^\ubul).
$$
Since $p$ is projective and $N^p$ is coherent on $X\times\bC^2$, $\bR^qp_*(N^p)$ is coherent on $\bC^2$. Therefore $\bR^kp_*(N^\ubul)$ is a coherent sheaf for every $k$.
\end{proof}

\begin{rmk}
$N^\ubul$ can be thought as a $\cc^2$ family of logarithmic $\lambda$-connections on $X$ with poles along $D$. Simpson has introduced in \cite{S} the Deligne-Hitchin twistor moduli space of logarithmic $\lambda$-connections $M_{\textrm{DH}}(X, \log D)$. $N^\ubul$ induces a map $\tau: \cc^2\to M_{\textrm{DH}}(X, \log D)$. Moreover, $N^\ubul$ is uniquely determined by the map $\tau$, because $N^\ubul$ is isomorphic to the pull-back of the universal logarithmic $\lambda$-connection on $X\times M_{\textrm{DH}}(X, \log D)$ by $\id\times \tau$. 
\end{rmk}

Since $\bR p_*(N^\ubul)$ has coherent cohomology by the lemma, one can form its cohomology support loci on $\bC^2$,
$$
\bV^i(\bR p_*(N^\ubul))\subset \bC^2.
$$
Indeed, recall from \cite[Section 2]{BW2} that for a bounded-above complex $M^\ubul$ of modules over a noetherian ring $R$, one has cohomology jump ideals $J^i_k(M^\ubul)\subset R$ defined as soon as $M^\ubul$ has finitely generated cohomology. The cohomology jump ideals $J^i_k(M^\ubul)$ define the cohomology jump loci $\bV^i_k(M^\ubul)$ in $Spec(R)$. The cohomology jump ideals are the same if $M^\ubul$ is replaced by a quasi-isomorphic complex. Moreover, if $M^\ubul$ is a complex of flat $R$-modules, the closed points of the cohomology jump loci $\bV^i_k(M^\ubul)$ are those maximal ideals $m\subset R$ for which $$\dim_{R/m}H^i(M^\ubul\otimes_R R/m)\ge k.$$ More generally:

\begin{lemma}\label{lemBW2}\cite[Corollary 2.4]{BW2} If $M^\ubul$ is a bounded-above a complex of flat $R$-modules with finitely generated cohomology and $S$ is an noetherian $R$-algebra, then $J^i_k(M^\ubul\otimes_R S)=J^i_k(M^\ubul)\cdot S$ for all $i, k$\footnote{The formula $J^i_k(E^\ubul)\otimes_R S=J^i_k(E^\ubul\otimes_R S)$ in \cite[Corollary 2.4]{BW2} should be $J^i_k(E^\ubul)\cdot S=J^i_k(E^\ubul\otimes_R S)$. }.
\end{lemma} 

Take for a representant of $\bR p_*(N^\ubul)$ a complex of flat $\cO_{\bC^2}$-modules. Hence, by base change (the derived categorical version of \cite[III, Proposition 9.3]{H}), $
\bV^k(\bR p_*(N^\ubul))$ consists of the closed points $(\eps,t)\in\bC^2$ for which $N^\ubul$ evaluated at these points has nontrivial $k$-th hypercohomology on $X$. In particular, we need to show that $\bV^k(\bR p_*(N^\ubul))$ does not contain the line $\eps=1$ for $k\le i$ and for $k\ge 2n-i$.

However, $\bV^k(\bR p_*(N^\ubul))$ is invariant under the diagonal $\bC^*$-action on $\bC^2$.  Under this action, the line $\eps=1$ sweeps out  all of $\bC^2\setminus\{\eps=0\}$. Therefore, since $\bV^k(\bR p_*(N^\ubul))$ must be closed in $\bC^2$, we only need to show that $\bV^k(\bR p_*(N^\ubul))$ is not equal to the whole $\bC^2$ for $k\le i$ and for $k\ge 2n-i$. Thus, we only need to show that the point $(\epsilon, t)=(0,1)$ is not contained in $\bV^k(\bR p_*(N^\ubul))$, which follows from the next lemma. 

\begin{lemma}
With assumptions as in Theorem \ref{thrm2vs}, we have:

(i) $\cH^k(\Omega_X^\ubul(\log D)\otimes_{\cO_X}\cL_\rho,w)=0$ for $k\le i$;

(ii) $H^p(X,\cH^q(\Omega_X^\ubul(\log D)\otimes_{\cO_X}\cL_\rho,w))=0$ for $p\ge n-i$.
\end{lemma}
\begin{proof} Recall that $\Omega_X^p(\log D)=\bigwedge^p\Omega_X^1(\log D)$. The ideal sheaf defining $Z(w)$ in $X$ is identified with the image of the last map in the complex, 
$$
w\wedge : \Omega_X^{n-1}(\log D)\longrightarrow \Omega_X^n(\log D)\cong \cO_X.
$$
Hence $$(\Omega_X^\ubul(\log D)\otimes_{\cO_X}\cL_\rho,w)=(\Omega_X^\ubul(\log D),w)\otimes_{\cO_X}\cL_\rho$$ is a Koszul complex for $Z(w)$. Recall that $w$ is a global section of $\Omega^1_X(\log D)$ such that $\codim_X(Z(w))>i$. Thus (i) follows from well-known properties of Koszul complexes since $X$ is smooth. Part (ii) follows from the fact the any cohomology sheaf of the Koszul complex has support in $Z(w)$.
\end{proof}

This finishes the proof of Theorem \ref{thrm2vs}, and hence that of Theorem \ref{thrmMain}.

\bigskip

\end{document}